\documentclass[preprint,12pt]{elsarticle}

\usepackage[T1]{fontenc}
\usepackage{lmodern}
\usepackage{microtype}
\usepackage{amsmath,amssymb,amsthm,mathtools}
\usepackage{enumitem}
\usepackage{xcolor}
\usepackage[colorlinks=true,citecolor=blue!55!black,linkcolor=blue!55!black,urlcolor=blue!55!black]{hyperref}

\biboptions{sort&compress}

\newtheorem{theorem}{Theorem}[section]
\newtheorem{lemma}[theorem]{Lemma}
\newtheorem{proposition}[theorem]{Proposition}
\newtheorem{corollary}[theorem]{Corollary}
\theoremstyle{definition}
\newtheorem{definition}[theorem]{Definition}
\theoremstyle{remark}
\newtheorem{remark}[theorem]{Remark}

\newcommand{\vc}{\operatorname{VC}}

\newcommand{\M}{\mathsf{M}}
\newcommand{\A}{\mathcal{A}}
\newcommand{\B}{\mathcal{B}}
\newcommand{\C}{\mathcal{C}}
\newcommand{\D}{\mathcal{D}}
\newcommand{\F}{\mathcal{F}}
\newcommand{\Hh}{\mathcal{H}}
\newcommand{\restr}[2]{#1\!\upharpoonright_{#2}}

\begin{document}

\begin{frontmatter}

\title{Recursive lower bounds for uniform set systems of bounded VC-dimension}

\author[1]{Xiaochen Zhao\corref{cor1}}
\ead{2250501013@cnu.edu.cn}

\author[1]{Gennian Ge\fnref{fund}}
\ead{gnge@zju.edu.cn}

\address[1]{School of Mathematical Sciences, Capital Normal University, Beijing, China}

\cortext[cor1]{Corresponding author.}

\fntext[fund]{Gennian Ge is supported by the National Key Research and Development Program of China under Grant 2025YFC3409900, the National Natural Science Foundation of China under Grant 12231014, and Beijing Scholars Program.}

\begin{abstract}
For integers $n\ge d+1$, let $\M_d(n)$ denote the maximum size of a
$(d+1)$-uniform family on an $n$-element ground set with VC-dimension at
most $d$.  For $n\ge2d+2$, the classical construction of Ahlswede and Khachatrian, later
generalized by Mubayi and Zhao, gives
\[
  \M_d(n)\ge \binom{n-1}{d}+\binom{n-4}{d-2}.
\]
We introduce a two-cover lifting construction and prove the recursive lower
bound
\[
  \M_d(n)\ge
  \binom{n-1}{d}+\binom{n-4}{d-2}+\M_{d-3}(n-5)
\]
for every $d\ge 3$ and $n\ge d+3$.  Consequently,
\[
  \M_d(n)\ge
  \binom{n-1}{d}+\binom{n-4}{d-2}+\binom{n-6}{d-3}.
\]
Thus the Mubayi--Zhao conjecture on the exact value of
$\M_d(n)$ for $n\ge2(d+2)$ is false for any $d\ge 3$.
The proof is elementary and proceeds entirely through an explicit analysis
of traces.
\end{abstract}

\begin{keyword}
VC-dimension \sep uniform set system \sep trace \sep extremal set theory
\MSC[2020] 05D05 \sep 05C65
\end{keyword}

\end{frontmatter}

\section{Introduction}

For a finite set $X$ and an integer $r\ge 0$, write
\[
  \binom{X}{r}=\{A\subseteq X:|A|=r\}.
\]
If $\F\subseteq 2^X$ and $S\subseteq X$, the trace of $\F$ on $S$ is
\[
  \restr{\F}{S}=\{F\cap S:F\in\F\}.
\]
The set $S$ is \emph{shattered} by $\F$ when
$\restr{\F}{S}=2^S$, and the VC-dimension of $\F$, denoted by
$\vc(\F)$, is the maximum cardinality of a shattered set.

For $n\ge d+1$, define
\[
  \M_d(n)=
  \max\left\{
    |\F|:
    \F\subseteq\binom{[n]}{d+1},\ 
    \vc(\F)\le d
  \right\}.
\]
Frankl and Pach~\cite{FranklPach} proved the algebraic upper bound
\[
  \M_d(n)\le \binom{n}{d}.
\]
Ahlswede and Khachatrian~\cite{AhlswedeKhachatrian} constructed a family
showing
\begin{equation}\label{eq:classical-lower}
  \M_d(n)\ge
  \binom{n-1}{d}+\binom{n-4}{d-2},
\end{equation}
and Mubayi and Zhao~\cite{MubayiZhao} produced many non-isomorphic
constructions of the same size and proposed that \eqref{eq:classical-lower}
is sharp for sufficiently large $n$.

Recent work has substantially improved the upper bound while continuing to
cite \eqref{eq:classical-lower} as the best general lower bound; see
Ge--Xu--Yip--Zhang--Zhao~\cite{GeXuYipZhangZhao},
Chao--Xu--Yip--Zhang~\cite{ChaoXuYipZhang}, and
Yang--Yu~\cite{YangYu}.  Wang, Xu, and Zhang~\cite{WangXuZhang} also
exhibited a larger isolated example in the $4$-uniform case on eight
vertices.

The purpose of this note is to give a recursive construction that strictly
improves \eqref{eq:classical-lower} for every $d\ge 3$.

\begin{theorem}\label{thm:main}
For every $d\ge 3$ and $n\ge d+3$,
\[
  \M_d(n)\ge
  \binom{n-1}{d}
  +\binom{n-4}{d-2}
  +\M_{d-3}(n-5).
\]
\end{theorem}

Taking any nonempty admissible residual family already gives a strict
improvement.  A convenient explicit consequence is the following.

\begin{corollary}\label{cor:explicit}
For every $d\ge 3$ and $n\ge d+3$,
\[
  \M_d(n)\ge
  \binom{n-1}{d}
  +\binom{n-4}{d-2}
  +\binom{n-6}{d-3}.
\]
\end{corollary}

The construction has two stages.  First, a pair of $d$-uniform families
covers the complete $d$-uniform layer, and its overlap is arranged so that
no overlap member is shattered by either family.  Second, two new vertices
lift this covering pair to a $(d+1)$-uniform family.  The classical numerical
bound is recovered by taking the residual family to be empty; the new term
comes from recursively enlarging the overlap.

\section{A uniform shattering criterion}

We record a standard observation that will be used repeatedly.

\begin{lemma}\label{lem:uniform-criterion}
Let $\F\subseteq\binom{X}{r}$.  Then $\vc(\F)\le r-1$ if and only if no
member of $\F$ is shattered by $\F$.
\end{lemma}

\begin{proof}
If some $F\in\F$ is shattered, then $\vc(\F)\ge |F|=r$.

Conversely, suppose $\vc(\F)\ge r$.  Since shattering is hereditary, there
is an $r$-set $S$ shattered by $\F$.  The full trace $S$ must be realized by
some $F\in\F$ with $F\cap S=S$.  As $|F|=|S|=r$, this forces $F=S$.
Thus a member of $\F$ is shattered.
\end{proof}

\section{A two-cover lifting lemma}

The following notion isolates the mechanism behind the construction.

\begin{definition}\label{def:admissible}
Let $W$ be finite and let $\A,\B\subseteq\binom{W}{d}$.  We call
$(\A,\B)$ an \emph{admissible covering pair} if
\[
  \A\cup\B=\binom{W}{d}
\]
and every $S\in\A\cap\B$ is shattered by neither $\A$ nor $\B$.
\end{definition}

For two new vertices $a,b\notin W$, define the lift
\begin{equation}\label{eq:lift}
\begin{split}
  \mathcal L(\A,\B)
  ={}&
  \bigl\{\{a,b\}\cup R:R\in\binom{W}{d-1}\bigr\}\\
  &{}\cup
  \bigl\{\{a\}\cup A:A\in\A\bigr\}
  \cup
  \bigl\{\{b\}\cup B:B\in\B\bigr\}.
\end{split}
\end{equation}

\begin{lemma}[Two-cover lifting]\label{lem:lifting}
If $(\A,\B)$ is an admissible covering pair on an $m$-element set $W$,
then
\[
  \mathcal L(\A,\B)\subseteq\binom{W\cup\{a,b\}}{d+1},
  \qquad
  \vc\bigl(\mathcal L(\A,\B)\bigr)\le d,
\]
and
\[
  \bigl|\mathcal L(\A,\B)\bigr|
  =
  \binom{m+1}{d}+|\A\cap\B|.
\]
\end{lemma}

\begin{proof}
Uniformity is immediate.  Put $\F=\mathcal L(\A,\B)$.  By
Lemma~\ref{lem:uniform-criterion}, it suffices to show that no member of
$\F$ is shattered.

First let $E=\{a,b\}\cup R$, where $R\in\binom{W}{d-1}$.  The trace $R$
is missing on $E$, since every member of $\F$ contains at least one of
$a,b$.

Next let $E=\{a\}\cup S$ with $S\in\A$.  A trace on $E$ that does not
contain $a$ can only be produced by a member $\{b\}\cup B$ with
$B\in\B$, and the resulting trace is $B\cap S$.  If $S\notin\B$, then
the full trace $S$ is absent from $\restr{\B}{S}$, because both $B$ and
$S$ have size $d$.  If $S\in\A\cap\B$, then $\B$ does not shatter $S$
by admissibility.  In either case, some trace not containing $a$ is
missing on $E$.  Hence $E$ is not shattered.

The case $E=\{b\}\cup S$ with $S\in\B$ is symmetric.

Finally,
\[
\begin{split}
  |\F|
  &=
  \binom{m}{d-1}+|\A|+|\B|\\
  &=
  \binom{m}{d-1}
  +\left|\binom{W}{d}\right|
  +|\A\cap\B|\\
  &=
  \binom{m+1}{d}+|\A\cap\B|,
\end{split}
\]
as required.
\end{proof}

\section{A recursively enlarged overlap}\label{sec:overlap}

We now construct an admissible covering pair with a large overlap.

Let $d\ge 3$, let
\[
  W=\{\alpha,\beta,\gamma\}\mathbin{\dot\cup}V,
\]
and let
\[
  \Hh\subseteq\binom{V}{d-2},
  \qquad
  \vc(\Hh)\le d-3.
\]
Partition $\binom{W}{d}$ into
\begin{align}
  \A_0
  &=
  \{S\in\binom{W}{d}:\alpha\notin S,\ \beta\in S\}
  \nonumber\\
  &\hspace{2cm}\cup
  \{S\in\binom{W}{d}:S\cap\{\alpha,\beta,\gamma\}=\varnothing\},
  \label{eq:A0}\\
  \B_0
  &=
  \{S\in\binom{W}{d}:\alpha\in S\}
  \nonumber\\
  &\hspace{2cm}\cup
  \{S\in\binom{W}{d}:\alpha,\beta\notin S,\ \gamma\in S\}.
  \label{eq:B0}
\end{align}
Thus $\A_0\mathbin{\dot\cup}\B_0=\binom{W}{d}$.

Define
\begin{align}
  \C_0
  &=
  \bigl\{
    \{\alpha,\beta\}\cup R:
    R\in\binom{W\setminus\{\alpha,\beta\}}{d-2}
  \bigr\},
  \label{eq:C0}\\
  \C_1
  &=
  \bigl\{
    \{\alpha,\gamma\}\cup Q:
    Q\in\Hh
  \bigr\},
  \label{eq:C1}\\
  \D
  &=
  \bigl\{
    \{\beta,\gamma\}\cup Q:
    Q\in\Hh
  \bigr\}.
  \label{eq:D}
\end{align}
Finally, put
\begin{equation}\label{eq:AB}
  \A=(\A_0\setminus\D)\cup\C_0\cup\C_1,
  \qquad
  \B=\B_0\cup\D.
\end{equation}

\begin{lemma}\label{lem:pair}
The pair $(\A,\B)$ defined in \eqref{eq:A0}--\eqref{eq:AB} is an
admissible covering pair, and
\[
  |\A\cap\B|
  =
  \binom{|W|-2}{d-2}+|\Hh|.
\]
\end{lemma}

\begin{proof}
The families $\A_0$ and $\B_0$ form a partition of $\binom{W}{d}$.
Moreover,
\[
  \C_0\cup\C_1\subseteq\B_0,
  \qquad
  \D\subseteq\A_0,
\]
and the three families $\C_0,\C_1,\D$ are pairwise disjoint.  Hence
\[
  \A\cup\B=\binom{W}{d},
  \qquad
  \A\cap\B=\C_0\cup\C_1.
\]
The asserted count follows immediately.  It remains to verify that every
member of the overlap is shattered by neither $\A$ nor $\B$.

Let
\[
  S=\{\alpha,\beta\}\cup R\in\C_0.
\]
First suppose that $\gamma\in R$.  We claim that
\[
  \{\alpha\}\notin\restr{\A}{S},
  \qquad
  \{\beta\}\notin\restr{\B}{S}.
\]
Indeed, members of $\A_0$ do not contain $\alpha$; members of $\C_0$
contain $\beta$; and members of $\C_1$ contain $\gamma$.  Thus no member
of $\A$ has trace $\{\alpha\}$ on $S$.  Similarly, a member of $\B_0$
either contains $\alpha$ or, if it avoids $\alpha$ and $\beta$, contains
$\gamma$; every member of $\D$ contains both $\beta$ and $\gamma$.
Thus no member of $\B$ has trace $\{\beta\}$ on $S$.

Now suppose that $\gamma\notin R$, so $R\subseteq V$.  Since
$|R|=d-2$ and $\vc(\Hh)\le d-3$, the set $R$ is not shattered by
$\Hh$.  Choose
\[
  J\subseteq R
  \quad\text{such that}\quad
  J\notin\restr{\Hh}{R}.
\]
We claim that
\[
  \{\alpha\}\cup J\notin\restr{\A}{S},
  \qquad
  \{\beta\}\cup J\notin\restr{\B}{S}.
\]
For the first claim, members of $\A_0$ avoid $\alpha$, members of
$\C_0$ contain $\beta$, and a member
$\{\alpha,\gamma\}\cup Q\in\C_1$ has trace
$\{\alpha\}\cup(Q\cap R)$ on $S$.  Such a trace cannot equal
$\{\alpha\}\cup J$ by the choice of $J$.  For the second claim,
members of the first part of $\B_0$ contain $\alpha$, members of the
second part of $\B_0$ avoid $\beta$, and a member
$\{\beta,\gamma\}\cup Q\in\D$ has trace
$\{\beta\}\cup(Q\cap R)$ on $S$.  Again the choice of $J$ rules out the
desired trace.  Therefore $S$ is shattered by neither family.

It remains to consider
\[
  S=\{\alpha,\gamma\}\cup Q\in\C_1,
  \qquad Q\in\Hh.
\]
We claim that
\[
  \{\gamma\}\cup Q\notin\restr{\A}{S},
  \qquad
  \varnothing\notin\restr{\B}{S}.
\]
For the first claim, a member of $\A_0$ with trace
$\{\gamma\}\cup Q$ would have to contain all $d-1$ elements of
$\{\gamma\}\cup Q$, avoid $\alpha$, and contain $\beta$.  The only
possibility is $\{\beta,\gamma\}\cup Q$, which belongs to $\D$ and was
deleted from $\A_0$.  Members of the second part of $\A_0$ avoid
$\gamma$, while every member of $\C_0\cup\C_1$ contains $\alpha$.
Thus the trace is absent.  Finally, every member of $\B$ contains at
least one of $\alpha,\gamma$: this is clear for both parts of $\B_0$
and for $\D$.  Hence the empty trace is absent from $\restr{\B}{S}$.

Thus every member of $\A\cap\B$ is shattered by neither $\A$ nor
$\B$, completing the proof.
\end{proof}

\section{Proof of the recursive bound}

\begin{proof}[Proof of Theorem~\ref{thm:main}]
Let $V$ be a set of size $n-5$, and choose
\[
  \Hh\subseteq\binom{V}{d-2}
\]
with
\[
  \vc(\Hh)\le d-3,
  \qquad
  |\Hh|=\M_{d-3}(n-5).
\]
Take three new vertices $\alpha,\beta,\gamma$ and set
\[
  W=\{\alpha,\beta,\gamma\}\mathbin{\dot\cup}V,
  \qquad |W|=n-2.
\]
Apply Lemma~\ref{lem:pair} to obtain an admissible covering pair
$(\A,\B)$ satisfying
\[
  |\A\cap\B|
  =
  \binom{n-4}{d-2}+\M_{d-3}(n-5).
\]
Now add two further vertices $a,b$ and apply
Lemma~\ref{lem:lifting}.  The resulting family is
$(d+1)$-uniform on an $n$-element set, having VC-dimension at most $d$ and size
\[
\begin{split}
  \binom{n-1}{d}+|\A\cap\B|
  &=
  \binom{n-1}{d}
  +\binom{n-4}{d-2}
  +\M_{d-3}(n-5).
\end{split}
\]
This proves the theorem.
\end{proof}

\begin{proof}[Proof of Corollary~\ref{cor:explicit}]
Fix $v\in V$ and take the star
\[
  \Hh=
  \{H\in\binom{V}{d-2}:v\in H\}.
\]
It is intersecting, and hence it cannot shatter any member of itself:
the empty trace on such a member would require a disjoint member.
Lemma~\ref{lem:uniform-criterion} therefore gives
$\vc(\Hh)\le d-3$.  Moreover,
\[
  |\Hh|=\binom{n-6}{d-3}.
\]
Substituting into Theorem~\ref{thm:main} proves the result.
\end{proof}

Theorem~\ref{thm:main} can be iterated.

\begin{proposition}\label{prop:iteration}
Let $t\ge 1$, $3t\le d$, and $n\ge d+2t+1$.  Then
\[
\begin{split}
  \M_d(n)\ge
  \sum_{j=0}^{t-1}
  \left[
    \binom{n-5j-1}{d-3j}
    +
    \binom{n-5j-4}{d-3j-2}
  \right]
  +\M_{d-3t}(n-5t).
\end{split}
\]
\end{proposition}

\begin{proof}
Apply Theorem~\ref{thm:main} successively to
\[
  (d,n),\ (d-3,n-5),\ldots,(d-3(t-1),n-5(t-1)).
\]
The numerical condition ensures that the hypothesis
$n-5j\ge(d-3j)+3$ holds at every step.
\end{proof}

\begin{remark}
Taking $\Hh=\varnothing$ in the construction of
Section~\ref{sec:overlap}~\eqref{eq:A0}--\eqref{eq:AB} yields the numerical lower bound
\eqref{eq:classical-lower}.  Thus the present construction enlarges the
classical overlap by the additional family $\C_1$, while the transfer of
$\D$ from $\A_0$ to $\B$ preserves the trace obstruction.
\end{remark}

\section{Concluding remarks}

The two-cover formulation suggests a separate extremal parameter.  For an
$m$-element set $W$, let
\[
  \mathsf{G}_d(m)
  =
  \max |\A\cap\B|,
\]
where the maximum is over admissible covering pairs
$\A,\B\subseteq\binom{W}{d}$.  Lemma~\ref{lem:lifting} gives
\[
  \M_d(m+2)\ge \binom{m+1}{d}+\mathsf{G}_d(m),
\]
while Lemma~\ref{lem:pair} yields
\[
  \mathsf{G}_d(m)
  \ge
  \binom{m-2}{d-2}+\M_{d-3}(m-3).
\]
Determining, or sharply bounding, $\mathsf{G}_d(m)$ appears to be a natural
way to organize further lower-bound constructions.

The best current general upper bound has the form
\[
  \M_d(n)\le \binom{n-1}{d}+O_d(n^{d-2})
\]
for fixed $d$ and large $n$~\cite{YangYu}.  Corollary~\ref{cor:explicit}
improves the lower bound by a term of order $n^{d-3}$ beyond the
Ahlswede--Khachatrian/Mubayi--Zhao construction.  Closing the remaining
gap requires additional control of the admissible overlap.

\bibliographystyle{plain}
\bibliography{references}

\end{document}